\documentclass[10pt]{article}
\usepackage[affil-it]{authblk}
\usepackage{amsmath}
\usepackage{amssymb}
\usepackage{amsthm}
\usepackage{amsfonts}
\usepackage[top=1.25in, left=1.2in, right=1.2in, bottom=1.5in]{geometry}

\theoremstyle{plain}
\newtheorem{theorem}{Theorem}[section]
\newtheorem{proposition}[theorem]{Proposition}
\newtheorem{lemma}[theorem]{Lemma}
\newtheorem{cor}[theorem]{Corollary}

\theoremstyle{definition}
\newtheorem{remark}[theorem]{Remark}

\title{On the Law of Large Numbers for Discrete Fourier Transform}
\author{Na Zhang}
\affil{ University of Cincinnati, PO Box 211005, Cincinnati, OH 45221-0025\\
Email: zhangn4@mail.uc.edu}

\date{}

\begin{document}

\maketitle

\begin{abstract}
We establish the rate of convergence in the strong law of large numbers of discrete Fourier Transform of the identically distributed random variables with finite $pth$ moment where $1<p<2$. Moreover, under an even weaker condition, i.e. $P(|X_n|\geq x)\leq P(|X_1|\geq x)$ for all $x\geq 0$ and the random variable $X_1$ has finite  $pth$ moment,  our results still hold.
\end{abstract}

%%%%%%%%%%%%%%%%%%%%%%%%%%%%%%%%%%%%%%%%%%%%%%%%%%%%%%%%%%%%%%%%

\section{Introduction and Results}

The law of large numbers is valid for pairwise independent random variables, result due to Etemadi (1981). This is a surprising result since a sequence of pairwise independent identically distributed random varibles may not be ergodic. A way to look into the speed of convergence of this result when the variables have finite moments of order $r$, $1<r<2$, is provided by Baum and Katz (1964) in the i.i.d case. By carefully examining the proof in George Stoica (2011), we notice that the proof can be adapted to centered pairwise independent random variables and we can formulate the following result.

\begin{proposition}
Let $(X_n)_{n\geq 1}$ be pairwise independent with the same distribution.\\

(a) Assume $X_n\in L^1$, then 
$$\frac{S_n}{n}\rightarrow EX_1, \ \ P-\text{a.s.} \ \text{as} \ n\to\infty,$$
where $S_n=\sum_{k=1}^{n}X_k$.\\

(b) Assume $X_n\in L^p$, $1<p<2$, then for any $1\leq r\leq p$, 
$$\sum_{n=1}^{\infty} n^{p/r-2} P\left(|S_n|>\epsilon n^{1/r}\right) <\infty.$$
\end{proposition}

Furthermore, $\displaystyle \frac{S_n-ES_n}{n^{1/p}}\rightarrow 0$ by Korchevsky (2015).\\

The goal of our note is to study the law of large numbers for the discrete Fourier Transform of a sequence of identically distributed random variables and to show that, from some point of view, the variables have similar properties with pairwise independent random variables.

Let $(X_n)_{n\geq 1}$ denote a sequence of identically distributed real valued random variables on a probability space  $(\Omega, \mathcal{F}, P)$. For $-\pi\leq t< \pi$, define the discrete Fourier transform 
\begin{equation}\label{eq1}
S_n(t)=\sum_{k=1}^{n} e^{ikt}X_k.
\end{equation}

We shall establish an anologue of the Baum and Katz (1964) result for the discrete Fourier Transform.\\

Our results are the following:

\begin{theorem}\label{theorem1}
If $(X_n)_{n\geq 1}$ has finite first moment, then for almost all $t\in[-\pi,\pi)$,
$$\lim_{n\to\infty} \frac{S_n(t)}{n}=0, \ \ \ P-\text{a.s}.$$
\end{theorem}

 The following theorem describes the rate of convergence in the strong law of large numbers:

\begin{theorem}\label{theorem2}

Let  $1<p<2$, $1\leq r\leq p$. If $(X_n)_{n\geq1}$ has finite $pth$ moment, then  for every $\epsilon >0$ and for almost all $t\in[-\pi,\pi)$,  
\begin{equation}\label{eq2}
\sum_{n=1}^{\infty} n^{p/r-2} P[\max_{1\leq k\leq n}|S_k(t) |>\epsilon n^{1/r}]<\infty.
\end{equation}
\end{theorem}

\begin{cor}\label{cor1}
Under the assumption of Theorem \ref{theorem2}, for almost all $t\in[\pi,\pi)$,
$$\lim_{n\to\infty}\frac{S_n(t)}{n^{1/p}} = 0, \ \ \ P-\text{a.s.}$$
\end{cor}

\begin{remark}\label{remark1}
Theorem \ref{theorem1} still holds if we replace the identically distributed condition  with $P(|X_n|\geq x)\leq P(|X_1|\geq x)$ for all $x\geq 0$ and $E|X_1|<\infty$;  Theorem \ref{theorem2} and Corollary \ref{cor1} are also true if we only have $P(|X_n|\geq x)\leq P(|X_1|\geq x)$ for all $x\geq 0$ and $E|X_1|^p<\infty$ where $1<p<2$.
\end{remark}

\section{Proofs}

Throughout this whole paper, $C>0$ denotes a generic constant  which may take different values from line to line. 

In order to prove Theorem \ref{theorem1}, we shall establish one preparatory lemma first. We begin by a truncation argument.

%lemma 1
\begin{lemma}\label{lemma1}
Let $Y_k=X_kI\{|X_k|\leq k\}$ and $S_n^*(t)=\sum_{k=0}^{n}e^{ikt}Y_k$. Then, for all t in $[-\pi,\pi)$,
$$\lim_{n\to\infty}\left|\frac{1}{n}S_n(t)-\frac{1}{n}S_n^*(t)\right|=0, \ \ \ P-\text{a.s.} $$ 
\end{lemma}

% proof of lemma1
\begin{proof}
By the fact that the random variables have the same distribution, we obtain

$$\sum_{n=1}^{\infty} P(X_n\neq Y_n)=\sum_{n=1}^{\infty} P(|X_n|>n)\leq \int_{0}^{\infty} P(|X_1|>x)dx=E|X_1|<\infty.$$

By the Borel-Cantelli Lemma, we know $P(X_n\neq Y_n \ i.o.)=0$. That is, for almost all  $\omega\in\Omega$, $X_n(\omega)=Y_n(\omega)$, for all $n$ sufficiently large, say for all $n\geq m(\omega):=m$. Thus for all  $t\in[-\pi,\pi)$,

$$\left|\frac{1}{n}S_n(t)-\frac{1}{n}S_n^*(t)\right|=\frac{1}{n}\left|\sum_{k=0}^{n} e^{ikt}(X_k-Y_k)\right|\leq \frac{1}{n}\sum_{k=0}^{m}|e^{ikt}|\cdot |X_k-Y_k|=\frac{1}{n}\sum_{k=0}^{m}|X_k-Y_k|\rightarrow 0, \ \ \ P-\text{a.s.}$$

That is,

$$\lim_{n\to\infty}\left|\frac{1}{n}S_n(t)-\frac{1}{n}S_n^*(t)\right|=0, \ \ P-\text{a.s.}$$

\end{proof}

\subsection{Proof of Theorem \ref{theorem1}}
\begin{proof}
First, let us show that $\displaystyle \sum_{k=1}^{n} \frac{e^{ikt}Y_k}{k}$ converges a.s. 

From Durrett (page 64), we know that   $\displaystyle\sum_{k=1}^{\infty}\frac{E Y_k^2}{k^2}\leq 4E|X_1|<\infty$. 
So,  for almost all  $\omega\in\Omega$, $\displaystyle \sum_{k=1}^{\infty}\frac{Y_k^2(\omega)}{k^2}<\infty$. 

Then, by Carleson's Theorem (1966), 
for almost all $\omega\in\Omega$, $\displaystyle \sum_{k=1}^{n} \frac{e^{ikt} Y_k}{k}$ converges a.s. in $t$. That is, for almost all $\omega\in\Omega$,  there exists $I_\omega\subset [-\pi,\pi]$ with $\lambda(I_\omega)=2\pi$,  such that for all $t \in I_\omega$, $\displaystyle \sum_{k=1}^{n} \frac{e^{ikt} Y_k}{k}$ converges.\\

Let $\Omega_0=\{\omega: \displaystyle \sum_{k=1}^{\infty}\frac{Y_k^2(\omega)}{k^2}<\infty\}$, then $P(\Omega_0)=1$.

It is convenient to work on the product space $[-\pi,\pi)\times\Omega$ with product measure $\widetilde P:=\lambda\times P$ where $\lambda$ is the  Lebesgue measure on $[-\pi,\pi)$. $\text{Define} \ \ \ A=\{(w,t):\displaystyle \sum_{k=1}^{\infty} \frac{e^{ikt} Y_k}{k} \ \ \text{is convergent}\}\subset\Omega\times[-\pi,\pi)\}.$

Using Fubini Theorem,
\begin{align*}
\widetilde{P}(A)&=\int_{[-\pi,\pi]\times \Omega} I_A(\omega,t)d\widetilde{P}=\int_{\Omega} \int_{-\pi}^{\pi} I_A(\omega,t)\,d\lambda dP \\
&=\int_{\Omega_0} \int_{-\pi}^{\pi} I_A(\omega,t)\,d\lambda dP =\int_{\Omega_0} 2\pi dP \\
&=2\pi \\
&=\int_{-\pi}^{\pi} \int_{\Omega} I_A(\omega,t)\,dP d\lambda . 
\end{align*}
Thus, for almost all $t\in[-\pi,\pi)$, $\displaystyle \sum_{k=1}^{\infty} \frac{e^{ikt} Y_k}{k}$ converges almost surely in $\omega$.\\

Now by applying Kronecker Lemma (If $a_n\uparrow \infty$ and $\sum_{n=1}^{\infty} \left(x_n/a_n\right)$ converges, then $a_n^{-1}\sum_{m=1}^{n} x_m \rightarrow 0$), for almost all $t\in[-\pi,\pi)$, we obtain
 $$ \lim_{n\to\infty}\frac{1}{n}\sum_{k=1}^{n} e^{ikt} Y_k=\lim_{n\to\infty}\frac{1}{n}S_n^*(t)=0 \ \ \  P-a.s.$$ 
By Lemma \ref{lemma1}, it follows that for almost all $t\in[-\pi,\pi)$:\\
 $$\displaystyle \lim_{n\to\infty}\frac{1}{n}S_n(t)=\lim_{n\to\infty}\frac{1}{n}S_n^*(t)=0, \ \ \  P-a.s.$$

\end{proof}

\subsection{Proof of Theorem \ref{theorem2}}

\begin{proof}
Define the following variables:

$$X_k^{'}=e^{itk}X_kI\{|X_k|\leq n^{1/r}\}, \ \ X_k^{''}=e^{itk}X_kI\{|X_k|> n^{1/r}\}.$$

Clearly, $e^{itk}X_k=X_k^{'}+X_k^{''}$ and $\displaystyle S_n(t)=S_n^{'}(t)+S_n^{''}(t)$
where $\displaystyle S_n^{'}(t)=\sum_{k=1}^{n}X_k^{'}$ and $S_n^{''}=\sum_{k=1}^{n}X_k^{''}$.

By Markov's Inequality,
\begin{equation}\label{eq3}
P\left[\max_{1\leq k\leq n}|S_k^{'}(t)|>\epsilon n^{1/r}\right] \leq \frac{1}{\epsilon^2}n^{-2/r} E\left[\left(\max_{1\leq k\leq n}|S_k^{'}(t)|\right)^2\right], 
\end{equation}

and by the maximal inequality in Hunt and  Young (1974),
\begin{equation}\label{eq4}
\int_{-\pi}^{\pi}\max_{1\leq k\leq n}|S_k^{'}(t)|^2\lambda(dt)\leq C\sum_{k=1}^{n}|X_k^{'}|^2.
\end{equation}

Using Fubini's Theorem and properties (\ref{eq3}) and (\ref{eq4}), we obtain:

\begin{align*}
\widetilde P\left[\max_{1\leq k\leq n}|S_k^{'}(t)|>\epsilon n^{1/r}\right]
&=\int_{-\pi}^{\pi} P\left[\max_{1\leq k\leq n}|S_k^{'}(t)|>\epsilon n^{1/r}\right]\lambda(dt)\\
&\leq  \frac{1}{\epsilon^2}n^{-2/r} \int_{-\pi}^{\pi} E\left[\left(\max_{1\leq k\leq n}|S_k^{'}(t)|\right)^2\right] \lambda(dt)\\
&=  \frac{1}{\epsilon^2}n^{-2/r} E\left[ \int_{-\pi}^{\pi} \left(\max_{1\leq k\leq n}|S_k^{'}(t)|\right)^2 \lambda(dt)\right] \\
&\leq \frac{C}{\epsilon^2}n^{-2/r}E\left[\sum_{k=1}^{n}|X_k^{'}|^2\right]\\
&= \frac{C}{\epsilon^2}n^{-2/r}\sum_{k=1}^{n} E\left[(X_k)^2I(|X_k|\leq n^{1/r})\right].
\end{align*}

By George Stoica (2011) (page 912), we can get:
\begin{equation*}
\sum_{n=1}^{\infty} n^{p/r-2/r-2} \sum_{k=1}^{n} E\left[(X_k)^2I(|X_k|\leq n^{1/r})\right] \leq C \sup_{k\geq 1}E|X_k|^p<\infty,
\end{equation*}
Combing this result with our computation, we obtain:

\begin{equation}\label{eq5}
\sum_{n=1}^{\infty} n^{p/r-2} \widetilde P\left[\max_{1\leq k\leq n}|S_k^{'}(t)|>\epsilon n^{1/r}\right]<\infty.
\end{equation}

Again, by Markov's inequality, 
\begin{align*}
P\left[\max_{1\leq k\leq n}|S_k^{''}(t)|>\epsilon n^{1/r}\right] &\leq \frac{1}{\epsilon}n^{-1/r} E\left[\max_{1\leq k\leq n}|S_k^{''}(t)|\right]\\
&\leq \frac{1}{\epsilon}n^{-1/r}  E\left[\sum_{k=1}^n|X_k^{''}|\right]\\
&=\frac{1}{\epsilon}n^{-1/r}  \sum_{k=1}^n E\left[|X_k^{''}|\right]\\
&=\frac{1}{\epsilon}n^{-1/r}  \sum_{k=1}^n E\left[|X_k|I\{|X_k|>n^{1/r}\}\right].
\end{align*}

By George Stoica (2011) (page 912), we can get:
\begin{equation*}
\sum_{n=1}^{\infty} n^{p/r-1/r-2} \sum_{k=1}^{n} E\left[|X_k|I(|X_k|\leq n^{1/r})\right] \leq C \sup_{k\geq 1}E|X_k|^p<\infty.
\end{equation*}

Combining this result with our computation, we obtain:

\begin{equation}\label{eq6}
\sum_{n=1}^{\infty} n^{p/r-2}  P\left[\max_{1\leq k\leq n}|S_k^{''}(t)|>\epsilon n^{1/r}\right]
\leq C E|X_k|^p<\infty.
\end{equation}

Using Fubini's Theorem and  relation (\ref{eq6})
\begin{equation}\label{eq7}
\begin{split}
\sum_{n=1}^{\infty}n^{p/r-2} \widetilde P\left[\max_{1\leq k\leq n}|S_k^{''}(t)|>\epsilon n^{1/r}\right]
&=\sum_{n=1}^{\infty}n^{p/r-2} \int_{-\pi}^{\pi} P\left[\max_{1\leq k\leq n}|S_k^{''}(t)|>\epsilon n^{1/r}\right]\lambda(dt)\\
&= \int_{-\pi}^{\pi} \sum_{n=1}^{\infty}n^{p/r-2} P\left[\max_{1\leq k\leq n}|S_k^{''}(t)|>\epsilon n^{1/r}\right]\lambda(dt)\\
&\leq 2\pi C E|X_k|^p<\infty.
\end{split}
\end{equation}

Combining (\ref{eq5}) and (\ref{eq7}), we get:
\begin{align*}
& \sum_{n=1}^{\infty}n^{p/r-2} \widetilde P[\max_{1\leq k\leq n}|S_k(t)|>\epsilon n^{1/r}]\\
&\leq \sum_{n=1}^{\infty}n^{p/r-2} \widetilde P[\max_{1\leq k\leq n}|S_k^{'}(t)|>\frac{\epsilon}{2} n^{1/r}]+\sum_{n=1}^{\infty}n^{p/r-2} \widetilde P[\max_{1\leq k\leq n}|S_k^{''}(t)|>\frac{\epsilon}{2} n^{1/r}]<\infty.
\end{align*}

By Fubini's Theorem, we have:
 \begin{align*}
 &\sum_{n=1}^{\infty}n^{p/r-2} \widetilde P[\max_{1\leq k\leq n}|S_k(t)|>\epsilon n^{1/r}]\\
& =\sum_{n=1}^{\infty}n^{p/r-2} \int_{-\pi}^{\pi}\int_{\Omega}I\{\max_{1\leq k\leq n}|S_k(t)|>\epsilon n^{1/r}\}dP\lambda(dt)\\
&=\int_{-\pi}^{\pi} \sum_{n=1}^{\infty}n^{p/r-2} \int_{\Omega}I\{\max_{1\leq k\leq n}|S_k(t)|>\epsilon n^{1/r}\}dP\lambda(dt)\\
&=\int_{-\pi}^{\pi} \sum_{n=1}^{\infty}n^{p/r-2} P\left(\max_{1\leq k\leq n}|S_k(t)|>\epsilon n^{1/r}\right)\lambda(dt)<\infty.
\end{align*}

Thus, for almost all $t\in[-\pi,\pi)$,

$$\sum_{n=1}^{\infty}n^{p/r-2} P\left(\max_{1\leq k\leq n}|S_k(t)|>\epsilon n^{1/r}\right)<\infty.$$

\end{proof}

\subsection{Proof of Corolary \ref{cor1}}
\begin{proof}

By Theorem \ref{theorem2}, when $1< r=p< 2$, we have: for almost all $t\in [-\pi,\pi)$,

$$\sum_{n=1}^{\infty}n^{-1} P\left(\max_{1\leq k\leq n}|S_k(t)|>\epsilon n^{1/p}\right)<\infty,$$
which is equivalent to 
$$\sum_{N=1}^{\infty} P\left(\max_{1\leq k\leq 2^N}|S_k(t)|>\epsilon 2^{N/p}\right)<\infty.$$

Then, by Remark 1 in Dedecker and Merlevede (2006), for almost all $t\in [-\pi,\pi)$, 
$$S_n(t)/n^{1/p}\rightarrow 0, \ \ \ P-\text{a.s.} \ \ $$
\end{proof}

\section{Appendix: Details about Remark \ref{remark1}}

%Proof of Remark
\begin{proof}
First, we show that Theorem \ref{theorem1} is true.

Here we use the same truncation as before, i.e.  $Y_k=X_kI\{|X_k|\leq k\}$.

As $P(|X_n|\geq x) \leq P(|X_1|\geq x)$ and $E|X_1|<\infty$, then

$$E|X_n|=\int_{0}^{\infty} P(|X_n|\geq x)dx \leq\int_{0}^{\infty} P(|X_1|\geq x)dx=E|X_1|<\infty .$$

$$\sum_{n=1}^{\infty} P(X_n\neq Y_n)=\sum_{n=1}^{\infty} P(|X_n|>n)\leq\sum_{n=1}^{\infty}P(|X_1|>n)\leq \int_{0}^{\infty} P(|X_1|>x)dx=E|X_1|<\infty.$$

So we still have Lemma \ref{lemma1}. Then the proof of Theorem \ref{theorem1} works.

The proof of Theorem \ref{theorem2} still works under the assumption of Remark \ref{remark1} and Theorem \ref{theorem2} implies Corollary \ref{cor1}.

\end{proof}

\section{Acknowledgement}
This paper is partially supported by the NSF grant, DMS-1512936.

\end{document}